\documentclass[11pt]{amsart}
\usepackage{geometry}                
\usepackage{amsmath}
\usepackage{amscd}
\usepackage{amsthm}
\usepackage{graphicx}
\usepackage{amssymb}
\usepackage{epstopdf}
\theoremstyle{plain}
\newtheorem{Thm}{Theorem}

\newtheorem{Prop}[Thm]{Proposition}

\newtheorem{Lem}[Thm]{Lemma}

\theoremstyle{definition}
\newtheorem{Defn}[Thm]{Definition}
\newtheorem{Expl}[Thm]{Example}

\theoremstyle{Remark}

\numberwithin{equation}{section}

\title{Derived McKay correspondence for $GL(3,\mathbf{C})$}
\author{Yujiro Kawamata}
\begin{document}
\maketitle

\begin{abstract}
We prove that the equivariant derived category for a finite subgroup of $GL(3,\mathbf{C})$ 
has a semi-orthogonal decomposition 
into the derived category of a certain partial resolution, called a maximal $\mathbf{Q}$-factorial terminalization,  
of the corresponding 
quotient singularity and a relative exceptional collection. 
This is a generalization of a result of Bridgeland, King and Reid.
\end{abstract}


\section{introduction}

The main theorem of this paper is the following generalization of the classical derived 
McKay correspondence for $SL(2,\mathbf{C})$ to the case of $GL(3,\mathbf{C})$:

\begin{Thm}\label{main}
Let $G$ be a finite subgroup of $GL(3,\mathbf{C})$, and let $X = \mathbf{C}^3/G$ be the quotient variety.
Let $\pi: \mathbf{C}^3 \to X$ be the projection, and define a $\mathbf{Q}$-divisor $B$ on $X$ by
an equality $\pi^*(K_X+B) = K_{\mathbf{C}^3}$.
Let $V_j$ ($j=1,\dots, m$) be all the proper linear subspaces of $\mathbf{C}^3$ whose innertia subgroups
$I_j = \{g \in G \mid g \vert_{V_j} = \text{Id}\}$ are non-trivial and not contained in $SL(3,\mathbf{C})$, let 
$D_j = \{g \in G \mid g(V_j)=V_j \}$ be the corresponding decomposition groups, and let $G_j = D_j/I_j$
be the quotient groups acting on the $V_j$. 
Then there exist smooth affine varieties $Z_i$ ($i=1,\dots,l$) 
and a projective birational morphism $f: Y \to X$ from a normal variety with only terminal quotient singularities, 
called a {\em maximal} $\mathbf{Q}$-factorial terminalization for the pair $(X,B)$,
such that there are fully faithful functors $\Phi_i: D^b(\text{coh}(Z_i)) \to D^b(\text{coh}([\mathbf{C}^3/G]))$ 
for $i =1,\dots,l$ and $\Phi: D^b(\text{coh}(\tilde Y)) \to D^b(\text{coh}([\mathbf{C}^3/G]))$ 
with a semi-orthogonal decomposition
\[
D^b(\text{coh}([\mathbf{C}^3/G])) 
\cong \langle \Phi_1(D^b(\text{coh}(Z_1))), \dots,  \Phi_l(D^b(\text{coh}(Z_l))), \Phi(D^b(\text{coh}(\tilde Y))) \rangle
\]
where $[\mathbf{C}^3/G]$ is a quotient stack, $\tilde Y$ is the smooth Deligne-Mumford stack associated to $Y$, 
and one of the following hold for each $Z_i$ ($i = 1,\dots,l$):

\begin{enumerate}

\item[(0)] $\dim Z_i = 0$, and $Z_i \cong V_j$ is the origin.

\item[(1)] $\dim Z_i = 1$, $Z_i$ is a smooth rational affine curve, 
and there is a finite morphism $Z_i \to V_j/G_j$ to the quotient curve for some $j$.

\item[(2)] $\dim Z_i = 2$, and $Z_i \to V_j/G_j$ is the minimal resolution of singularities of 
the quotient surface for some $j$.

\end{enumerate}

The correspondence $\{1,\dots,l\} \to \{1,\dots,m\}$ given by $i \mapsto j$ is not necessarily injective nor surjective.
\end{Thm}

If $G \subset SL(3,\mathbf{C})$, then this is a theorem of Bridgeland-King-Reid \cite{BKR}.
In this case, $m=l=0$, $Y$ is smooth, and the conclusion is reduced to the equivalence
$D^b([\mathbf{C}^3/G]) \cong D^b(\text{coh}(Y))$.

We would like to call the terms on the right hand side except $\Phi_0(D^b(\text{coh}(Z_0)))$ 
a {\em relative exceptional collection} because it is an 
exceptional collection if $\dim Z_i = 0$ for all $i \ge 1$.

A maximal $\mathbf{Q}$-factorial terminalization is not unique, and the theorem holds for some choice of it.

We shall explain the terminology and the background of the theorem in the next section.
The theorem is a special case of so-called \lq\lq DK hypothesis''.

This work was partly done while the author stayed at National Taiwan University.
The author would like to thank Professor Jungkai Chen and 
National Center for Theoretical Sciences of Taiwan of the hospitality and 
excellent working condition.

This work is partly supported by Grant-in-Aid for Scientific Research (A) 16H02141.

\section{DK-hypothesis}

DK-hypothesis is a working hypothesis saying that equalities and inequalities of canonical divisors correspond to
equivalences and semi-orthogonal decompositions of derived categories (\cite{DK}, \cite{Seattle}).

We can compare $\mathbf{R}$-Cartier divisors on different but birationally equivalent normal varieties, i.e., 
$B = C$, or $B \ge C$, if the pull-backs of $B$ and $C$ coincide as divisors on some
sufficiently high model, or the pull-backs of $B$ is larger than that of $C$, respectively.
More precisely, suppose that there is a proper birational map $\alpha: X \dashrightarrow Y$ 
between normal varieties, 
and let $B$ and $C$ be $\mathbf{R}$-Cartier divisors on $X$ and $Y$, respectively.
Then there is a third normal variety $Z$ with proper birational morphisms $f: Z \to X$ and $g: Z \to Y$ such that 
$\alpha = g \circ f^{-1}$.
In this case, we say $B = C$ (resp. $B \ge C$) if $f^*B = g^*C$ (resp. $f^*B \ge g^*C$).
We note that this definition does not depend on the choice of the third model $Z$, but only on $\alpha$.

The canonical divisor $K_X$ of a normal algebraic variety $X$ is the basic invariant for the classification of varieties.
The minimal model program is the process consisting of birational maps which decrease the canonical divisors.
The canonical divisors of two projective varieties linked by a proper birational map $\alpha: X \dashrightarrow Y$ 
are compared by using the same rational canonical differential form of the function field.

In the minimal model program, it is better to deal with pairs instead of varieties.
This is so-called \lq\lq log'' philosophy.
Let $(X,B)$ be a pair of a normal variety $X$ and an effective $\mathbf{R}$-divisor $B$ 
such that $K_X+B$, the {\em log canonical divisor}, is an $\mathbf{R}$-Cartier divisor.
A prime divisor $E$ {\em above} $X$ is the one which appears as a prime divisor on a normal variety $Y$ 
with a birational morphism $f: Y \to X$.
The {\em coefficient} $e$ of $E$ is defined to be 
the coefficient in the expression $f^*(K_X+B) = K_Y + eE + \text{other components}$.
$-e$ is called the {\em discrepancy} of $E$ and $1-e$ the {\em log discrepancy}.
The coefficient does not depend on the choice of a higher model $Y$ but only on the discrete valuation of 
$k(X)$ determined by $E$.
The pair is called {\em KLT} if $e < 1$ for any prime divisor above $X$.

It is important to note that the canonical divisor of a smooth projective variety is a categorical invariant.  
Let $D^b(\text{coh}(X))$ be the bounded derived category of a smooth projective variety $X$.
Then the {\em Serre functor} $S_X$ is given by $S_X(a) = a \otimes \omega_X[\dim X]$ 
for any $a \in D^b(\text{coh}(X))$, where $\omega = \mathcal{O}_X(K_X)$ is the canonical sheaf. 
Serre functor satisfies a bifunctorial formula 
$\text{Hom}(a,b)^* \cong \text{Hom}(b,S_X(a))$ for $a,b \in D^b(\text{coh}(X))$.
Here we note that $\text{Hom}(a,b)$ is a finite dimensional vector space over the base field, 
and the functor $S_X$ is uniquely determined by this formula.

In general, let $\mathcal{D}$ be a triangulated category, and let $\mathcal{A}_i$ ($i = 1,\dots,l$) 
be triangulated full subcategories.
We say that there is a {\em semi-orthogonal decomposition (SOD)} 
$\mathcal{D} = \langle \mathcal{A}_1, \dots, \mathcal{A}_l \rangle$ if the following conditions are satisfied:

\begin{enumerate}

\item For any object $d \in \mathcal{D}$, there exist objects $a_i \in \mathcal{A}_i$ and
distinguished triangles $a_{i+1} \to d_{i+1} \to d_i \to a_{i+1}[1]$ ($i = 1,\dots,l-1$) 
for some $d_i \in \mathcal{D}$ such that $d_1 = a_1$ and $d_l = d$. 

\item $\text{Hom}(a_i,a_j) = 0$ if $a_i \in \mathcal{A}_i$, $a_j \in \mathcal{A}_j$, and $i > j$.

\end{enumerate}

{\em DK-hypothesis} is a working hypothesis stating that the inequalities of the canonical divisors $K_X$, or
more generally the log canonical divisors $K_X+B$, 
and the SOD's of the suitably defined derived categories $D(X)$ or $D(X,B)$ are parallel; the following
should be equivalent:

\begin{itemize}

\item $K_X = K_Y$ (resp. $K_X \ge K_Y$), or $K_X+B=K_Y+C$ (resp. $K_X+B \ge K_Y+C$).

\item $D(X) \cong D(Y)$ (resp. $D(X) = \langle \mathcal{A}, D(Y) \rangle$ for some $\mathcal{A}$), or
$D(X,B) \cong D(Y,C)$ (resp. $D(X,B) = \langle \mathcal{A}, D(Y,C) \rangle$ for some $\mathcal{A}$).

\end{itemize}

If $X$ is a smooth projective variety, then the bounded derived category of 
coherent sheaves$D(X) = D^b(\text{coh}(X))$ works well.
But the definition of $D(X)$ should be modified according to the situations.
Indeed the minimal model program requires that we consider pairs $(X,B)$ with singularities.
If $X$ has only quotient singularities, then we should consider the smooth Deligne-Mumford
stack associated to $X$ as manifested in \cite{Francia}.
If we consider a pair $(X,B)$, then we should define $D(X,B)$ suitably.
There is still ambiguity.

\vskip 1pc

This parallelism between MMP and SOD are first observed by Bondal and Orlov in the following example (\cite{BO}):

\begin{Expl}
(1) Let $f: X \to Y$ be a blowing up of a smooth projective variety $Y$ with a smooth center $F$ of codimension $c$ and
the exceptional divisor $E$.
Then $K_X = f^*K_Y + (c-1)E$.
Correspondingly, there is an SOD:
\[
\begin{split}
&D^b(\text{coh}(X)) \\
&= \langle j_*g^*D^b(\text{coh}(F)) \otimes \mathcal{O}_E((-c+1)E), \dots,
j_*g^*D^b(\text{coh}(F)) \otimes \mathcal{O}_E(-E), f^*D^b(\text{coh}(Y)) \rangle
\end{split}
\]
where $j: E \to X$ and $g: E \to F$ are natural morphisms.

(2) Let $X$ be a smooth projective variety of dimension $m+n+1$ containing a subvariety $E \cong \mathbf{P}^m$ whose
normal bundle is isomorphic to $\mathcal{O}_{\mathbf{P}^m}(-1)^{\oplus n+1}$.
Assume that $m \ge n$.
Let $f: Z \to X$ be the blowing up with center $E$ with the exceptional divisor
$G \cong \mathbf{P}^m \times \mathbf{P}^n$, 
and let $g: Z \to Y$ be the blowing down of $G$ to another direction.
Let $F = g(G) \cong \mathbf{P}^n$.
Then $f^*K_X = g^*K_Y + (m-n)G$.
Correspondingly, there is an SOD:
\[
D^b(\text{coh}(X)) = \langle j_*\mathcal{O}_E(-m+n), \dots,j_*\mathcal{O}_E(-1), f_*g^*D^b(\text{coh}(Y)) \rangle
\]
where $j: E \to X$ is a natural morphisms.
\end{Expl}

We note that (1) (resp. (2)) is a special case of a divisorial contraction (resp. flip) in the MMP.

If we take the unbounded derived category of quasi-coherent sheaves $D(\text{Qcoh}(X))$, 
then the DK-hypothesis does not work well:

\begin{Expl}\label{qcoh}
Let $f: X \to Y$ be a projective birational morphism of normal varieties such that 
$Rf_*\mathcal{O}_X \cong \mathcal{O}_Y$.
Then the pull-back functor $Lf^*: D(\text{Qcoh}(Y)) \to D(\text{Qcoh}(X))$ is fully-faithful, and 
there is a semi-orthogonal decomposition 
$D(\text{Qcoh}(X)) = \langle \mathcal{C}, Lf^*D(\text{Qcoh}(Y)) \rangle$ for 
$\mathcal{C} = \{c \in D(\text{Qcoh}(X)) \mid Rf_*c = 0\}$.

For example, let a cyclic group $G \cong \mathbf{Z}/r\mathbf{Z}$ act on an affine space 
$M = \mathbf{C}^n$ diagonally with the same weight, 
and let $Y = M/G$ be the quotient space.
It has a quotient singularity of type $1/r(1,\dots,1)$, and 
the blowing up at the singular point $f: X \to Y$ is a resolution of singularities.
Let $E \cong \mathbf{P}^{n-1}$ be the exceptional locus of $f$.
Then we have $K_X = f^*K_Y + (n-r)/r E$.
Thus $K_X < K_Y$ if $n < r$, but the unbounded derived categories of quasi-coherent sheaves
are related by fully faithful inclusions of opposite direction. 

Therefore the unbounded derived category of quasi-coherent sheaves does not detect the level of the canonical divisor
when there are singularities.
We note that quotient singularities, which are considered as pairs with $0$ boundaries, are typical KLT singularities, 
the singularities which appear naturally in the minimal model program
(KLT pair is defined in the next section).
In this sense, the category $D(\text{Qcoh}(X))$ is too big to answer subtle questions such as the MMP. 
\end{Expl}

The DK hypothesis works well if we replace a variety having quotient singularities by a Deligne-Mumford stack:

\begin{Expl}
We continue Example~\ref{qcoh}.
We consider the quotient stack $\tilde Y = [M/G]$ instead of the quotient variety $Y = M/G$. 
Then a coherent sheaf on $\tilde Y$ is nothing but a $G$-equivariant coherent sheaf on $M$.
Thus we have an equivalence $D^b(\text{coh}(\tilde Y)) \cong D^b(\text{coh}^G(M))$.

Let $g: M' \to M$ be the blowing-up at the origin $Q \in M$.
Then $G$ acts on $M'$ and we have an isomorphism $M'/G \cong X$.
Let $h: M' \to X$ be the projection.
Let $\mathcal{O}_Q(i) \in D^b(\text{coh}^G(M))$ ($0 \le i < r$) be a skyscraper sheaf of length $1$ 
supported at the origin $Q \in M$ on which $G$ acts by weight $i$.

If $n \le r$, then a functor $\Phi: D^b(\text{coh}(X)) \to D^b(\text{coh}^G(M))$ defined by 
$\Phi(b) = Rg_*Lh^*(b)$ is fully faithful, and we have a semi-orthogonal decomposition
\[
D^b(\text{coh}^G(M)) = \langle \mathcal{O}_Q(r-n),\dots, \mathcal{O}_Q(1), \Phi(D^b(\text{coh}(X))) \rangle.
\]
If $n \ge r$, then a functor $\Psi: D^b(\text{coh}^G(M)) \to D^b(\text{coh}(X))$ defined by
$\Psi(a) = (Rh_*Lg^*(a))^G$ is fully faithful, and we have a semi-orthogonal decomposition
\[
D^b(\text{coh}(X)) = \langle \mathcal{O}_E(-n+r),\dots, \mathcal{O}_E(-1), \Psi(D^b(\text{coh}^G(M))) \rangle.
\]
\end{Expl}

The derived McKay correspondence is a special case of DK-hypothesis:

\begin{Thm}[\cite{GSV}, \cite{KV}]
Let $G \subset SL(2,\mathbf{C})$ be a finite subgroup acting naturally on $M = \mathbf{C}^2$, 
let $Y = M/G$ be the quotient space, and let $f: X \to Y$ be the minimal resolution of singularities.
Let $M' = X \times_Y M$ be the fiber product with projections $p: M' \to X$ and $q: M' \to M$.
Then the functor $\Phi = Rq_*Lp^*: D^b(\text{coh}(X)) \to D^b(\text{coh}^G(M))$ is an equivalence.
\end{Thm}

The derived McKay correspondence is extended to the case $G \subset SL(3,\mathbf{C})$ (\cite{BKR}) and 
$G \subset GL(2,\mathbf{C})$ (\cite{Ishii-Ueda},\cite{toricIII}).
The purpose of this paper is to extend it to the case $G \subset GL(3,\mathbf{C})$.

An important special case of {\em $K$-equivalence}, a proper birational map which does not change the level of 
canonical divisors, is a flop.
A {\em flop} is a diagram as follows
\[
\begin{CD}
X @>f>> Y @<g<< Z
\end{CD}
\]
where $f$ and $g$ are projective birational morphisms which are isomorphisms in codimension $1$, 
$\rho(X/Y) = \rho(Z/Y) = 1$ (relative Picard numbers), and
such that $K_X = f^*K_Y$ and $K_Z = g^*K_Y$.
The DK-hypothesis predicts that $D(X) \cong D(Y)$ for suitable interpretations.
There are positive answers in many cases (\cite{Bridgeland}, \cite{Chen}, \cite{VdBergh}, \cite{Namikawa}, 
\cite{DK}, \cite{Kaledin}, 
\cite{BK}, \cite{Cautis}, \cite{DS}, \cite{BFK}, \cite{HL}, \cite{CKP}).

The following proposition supports the DK-hypothesis in the case of inequalities:

\begin{Prop}
Let $X$ and $Y$ be smooth projective varieties.
Assume the following conditions:

(1) There is a fully faithful functor $j_*: D^b(\text{coh}(X)) \to D^b(\text{coh}(Y))$ 
with a kernel $P \in D^b(\text{coh}(X \times Y)$.

(2) There are an irreducible component $Z$ of the support of $P$ and 
open dense subsets $X^o, Y^o, Z^o$ of $X,Y,Z$, respectively, 
such that 
$p_1: Z^o \to X^o$ and $p_2: Z^o \to Y^o$ are isomorphisms and 
$P \vert_{X^o \times Y^o} \cong O_{\Delta_{X^o}}$.

Then there is an inequality $p_1^*K_X \le p_2^*K_Y$ on $Z$.
\end{Prop}

\begin{proof}
Let $j^!$ be the right adjoint functor of $j_*$.
There is an adjunction morphism $j_*j^! \to 1$.
Let $S_X$ and $S_Y$ be the Serre functors on $D^b(\text{coh}(X))$ and $D^b(\text{coh}(Y))$, respectively.
Then we have $S_X = j^!S_Yj_*$.
Thus there is a morphism of functors $j_*S_X = j_*j^!S_Yj_* \to S_Yj_*$.

The kernels of $j_*S_X$ and $S_Yj_*$ are given by $p_1^*\omega_X \otimes P$ and
$P \otimes p_2^*\omega_Y$, respectively.
Therefore there is a morphism 
\[
f: p_1^*\omega_X \otimes P \to P \otimes p_2^*\omega_Y.
\]
Since $f \vert_{X^o \times Y^o}$ is an isomorphism, $f$ gives an inequality $p_1^*K_X \le p_2^*K_Y$.
\end{proof}

We have similar statement for smooth Deligne-Mumford stacks associated to 
projective varieties with only quotient singularities, 

\section{toroidal case}

We recall results on the change of derived categories under the toroidal MMP.

Let $(X,B)$ be a pair consisting of a toroidal $\mathbf{Q}$-factorial variety $X$ whose boundary has no self-intersection
and a toroidal $\mathbf{Q}$-divisor $B$ whose coefficients 
belong to the standard set $\{1 - 1/e \mid e \in \mathbf{Z}_{>0}\}$.
We call such a pair simply a {\em toroidal $\mathbf{Q}$-factorial KLT pair} in this paper.

For a toroidal $\mathbf{Q}$-factorial KLT pair $(X,B)$, the variety $X$ has only quotient singularities,
 i.e., there is a quasi-finite and surjective morphism $\pi': X' \to X$ 
from a smooth scheme, which may be reducible, such that $\pi'$ is \'etale over the smooth locus of $X$, 
and that the pull-back $B' = (\pi')^*B$ is supported by a simple normal crossing divisor. 
Then there is a quasi-finite surjective morphism $\pi'': X'' \to X'$ from a smooth scheme whose ramification is 
given by the formula $(\pi'')^*(K_{X'} + B') = K_{X''}$. 
The covering morphism $\pi' \circ \pi'': X'' \to X$ defines a structure 
of a smooth Deligne-Mumford stack $\tilde X$ with a bijective morphism $\pi: \tilde X \to X$
such that $\pi^*(K_X+B) = K_{\tilde X}$.
We call $\tilde X$ the smooth Deligne-Mumford stack {\em associated to the pair} $(X,B)$. 
It is independent of the choice of $\pi'$ and $\pi''$ but only on the pair $(X,B)$.
Toroidal singularities are abelian quotient singularities.
But this construction works for any quotient singularities.

We start with the case of Mori fiber spaces:

\begin{Thm}[\cite{toric}]
Let $(X,B)$ be a toroidal $\mathbf{Q}$-factorial KLT pair, 
let $Y$ be another toroidal variety whose boundary has no self-intersection, 
and let $f: X \to Y$ be a projective surjective toroidal
morphism with connected fibers such that $\dim Y < \dim X$, $\rho(X/Y) = 1$ for the relative Picard number, 
and that $-(K_X+B)$ is $f$-ample.
Then the following hold:

(1) There exists a toroidal $\mathbf{Q}$-divisor $C$ on $Y$ such that 
$(Y,C)$ is a toroidal $\mathbf{Q}$-factorial KLT pair, 
and $f$ induces a smooth morphism
$\tilde f: \tilde X \to \tilde Y$ between smooth Deligne-Mumford stacks associated 
to the pairs $(X,B)$ and $(Y,C)$.

(2) There exist an integer $l$ and 
fully faithful functors $\Phi_i: D^b(\text{coh}(\tilde Y)) \to D^b(\text{coh}(\tilde X))$ for $i=1,\dots,l$ such that 
there is a semi-orthogonal decomposition:
\[
D^b(\text{coh}(\tilde X)) = \langle \Phi_1(D^b(\text{coh}(\tilde Y))), \dots, \Phi_l(D^b(\text{coh}(\tilde Y))) \rangle.
\] 
\end{Thm}

It is natural to consider the pairs instead of varieties because we have to consider them.
Indeed $C$ can be non-zero even if $B = 0$.
$\tilde X$ is smooth even if $X$ is singular, and $\tilde f$ is smooth even if $f$ has singular fibers.

The DK hypothesis works well in the case of toroidal MMP if we replace varieties and pairs by their associated stacks:

\begin{Thm}[\cite{logcrep}, \cite{toric}, \cite{toricII}, \cite{toricIII}]\label{toric}
Let $(X,B)$ and $(Y,C)$ be toroidal $\mathbf{Q}$-factorial KLT pairs, and 
let $\tilde X$ and $\tilde Y$ be the smooth Deligne-Mumford stacks associated to them.
Assume that one of the following conditions hold:

\begin{enumerate}

\item $X = Y$,  $B \ge C$, and the support of $B-C$ is a prime divisor.

\item There is a toroidal projective birational morphism $f: X \to Y$ whose exceptional locus is a prime divisor 
and such that $K_X + B \ge f^*(K_Y+C)$ and $C = f_*B$.

\item There is a toroidal projective birational morphism $f: Y \to X$ whose exceptional locus is a prime divisor 
and such that $f^*(K_X + B) \ge K_Y+C$ and $B = f_*C$.

\item There is a toroidal birational map $f: X \dashrightarrow Y$ which is factorized as $f = f_2^{-1} \circ f_1$ by
toroidal projective birational morphisms $f_1: X \to Z$ and $f_2: Y \to Z$ whose exceptional 
loci have codimension at least 2 and such that $\rho(X/Z) = \rho(Y/Z) = 1$.
Moreover $\alpha^*(K_X + B) \ge \beta^*(K_Y+C)$ for some projective 
birational morphisms $\alpha: W \to X$ and $\beta: W \to Y$ from another variety $W$.

\end{enumerate}

Then there is a toroidal $\mathbf{Q}$-factorial KLT pair $(F,B_F)$ with an integer $l$ in the cases except (4)
(resp. toroidal $\mathbf{Q}$-factorial KLT pairs $(F_j,B_{F_j})$ with integers $l_j$ for $j = 1,\dots,m$ in the case (4))
such that there are fully faithful functors 
$\Phi_i: D^b(\text{coh}(\tilde F)) \to D^b(\text{coh}(\tilde X))$ ($i=1,\dots,l$) 
(resp. $\Phi_{i,j}: D^b(\text{coh}(\tilde F_j)) \to D^b(\text{coh}(\tilde X))$ ($i=1,\dots,l_j$, $j=1,\dots,m$)))
and 
$\Phi: D^b(\text{coh}(\tilde Y)) \to D^b(\text{coh}(\tilde X))$ 
with a semi-orthogonal decomposition
\[
D^b(\text{coh}(\tilde X)) = \langle \Phi_1(D^b(\text{coh}(\tilde F))), \dots, \Phi_l(D^b(\text{coh}(\tilde F))), 
\Phi(D^b(\text{coh}(\tilde Y))) \rangle
\]
(resp. 
$D^b(\text{coh}(\tilde X)) = \langle \Phi_{i,j}(D^b(\text{coh}(\tilde F_j)))_{i,j},
\Phi(D^b(\text{coh}(\tilde Y))) \rangle$), 
where $\tilde F$ (reps. $\tilde F_j$) is the smooth Deligne-Mumford stack associated to the pair $(F,B_F)$
(resp. $(F_j,B_{F_j})$).
Moreover the pairs $(F,B_F)$ and $(F_j,B_{F_j})$ appear in the following way in the cases (1) through (4):

\begin{enumerate}
\item $F$ is the support of $B-C$.

\item For the exceptional divisor $E$ of $f$, the induced morphism $f \vert_E: E \to f(E)$ is a Mori fiber space,
and $F = f(E)$.

\item The same as the case (2).

\item For the connected components $E_j$ of the exceptional locus of $f_1$, 
the induced morphisms $f_1 \vert_{E_j}: E_j \to f_1(E_j)$ are Mori fiber spaces,
and $F_j = f_1(E_j)$.

\end{enumerate}

The boundary divisors $B_F$ and $B_{F_j}$ are determined by suitable adjunctions.
\end{Thm}

\begin{proof}
The proofs in \cite{toric} and \cite{toricII} on the toric case can be extended to the toroidsl case, 
because the fully faithfulness of functors can be checked locally by \cite{Bridgeland1}~Theorem~2.3.

In the case (4), we note that the exceptional locus of $f_1$ is irreducible in the toric case, but 
it is not necessary the case in the toroidal case.
Therefore we need possibly more than one $F_j$'s. 

In the case (1), the semi-orthogonal complement of the image $\Phi(D^b(\text{coh}(\tilde Y)))$ is not described 
in \cite{logcrep}~Theorem~4.2.
So we do this here.

We use the notation of loc cit.
We have $B = \sum (1-1/r_i)D_i$, $C = \sum (1-1/s_i)D_i$, and $r_i \ge s_i$ with equalities for
$i \ne 1$.
We set $F = D_1$.
A toroidal line bundle on $\tilde Y$ is in the form $L(m) = \mathcal{O}_{\tilde Y}(m)$ with $m = \sum (m_i/s_i)v_i^*$, 
where $\{v_i^*\}$ is the basis of monomials on $X=Y$ which are dual to the vectors $\{v_i\}$ 
corresponding to the divisors $\{D_i\}$.
Thus $\mathcal{O}_{\tilde Y}(m) = \mathcal{O}_{\tilde Y}(- \sum (m_i/s_i)D_i) = 
\mathcal{O}_{\tilde Y}(- \sum m_i \tilde D_{Y,i})$,
where the $\tilde D_{Y,i}$ are prime divisors on $\tilde Y$ above the $D_i$.
We have $H^0(\tilde Y, \mathcal{O}_{\tilde Y}(m)) \cong H^0(Y, \mathcal{O}_Y(-\sum \ulcorner m_i/s_i \urcorner D_i))$.
We have similar expressions on $\tilde X$ and $\tilde D_{X,i}$.

The image of $L(m)$ by $\Phi$ is a line bundle $\Phi(L(m)) = \mathcal{O}_{\tilde X}(\phi(m))$ on $\tilde X$ with 
$\phi(m) = \sum (\ulcorner m_ir_i/s_i \urcorner /r_i)v_i^*$. 
It is already proved that $\Phi$ is a fully faithful functor, because
$\ulcorner (m_i - m'_i)/s_i \urcorner
= \ulcorner (\ulcorner m_ir_i/s_i \urcorner - \ulcorner m'_ir_i/s_i \urcorner)/r_i \urcorner$.

The lattice corresponding to the toroidal variety $F = D_1$ is given by $N_F = N_X/\mathbf{Z}v_1$, where $N_X$ is the
lattice for $X$.
Let $t_i$ be the integers such that $v_i \text{ mod } \mathbf{Z}v_1 = t_i\bar v_i$ for $i \ne 1$, 
where the $\bar v_i$ are primitive vectors in $N_F$.
Then we define the boundary $\mathbf{Q}$-divisor on $F$ by $B_F = \sum_{i \ne 1} (1-1/r_it_i)D_{F,i}$ 
with the set theoretic intersection $D_{F,i} = D_i \cap F$. 
Let $\tilde F$ be the Deligne-Mumford stack associated to the pair $(F, B_F)$.
Let $\tilde D_{F,i}$ be prime divisors on $\tilde F$ above the $D_{F,i}$ for $i \ne 1$.

A line bundle on $\tilde F$ is of the form $L_F(m) = \mathcal{O}_{\tilde F}(m)$ 
with $m = \sum_{i \ne 1} (m_i/r_it_i)\bar v_i^*$, 
where $\{\bar v_i^*\}$ is the basis of monomials on $F$ which are dual to the vectors $\{\bar v_i\}$ 
corresponding to the divisors $\{D_{F,i}\}$.
Let $\tilde E = \tilde F \times_X \tilde X$ be the fiber product with projections
$\alpha: \tilde E \to \tilde F$ and $\beta: \tilde E \to \tilde X$.

We give more geometric descriptions of the above local models.
Let $M = \mathbf{C}^n$, and let an abelian group $G$ act on $M$ faithfully and diagonally, such that 
$X = \mathbf{C}^n/G$ is the quotient variety.
The associated stack $\tilde X = [\mathbf{C}^n/G]$ is a global quotient stack.
The divisors $D_i$ are the images of the coordinate hyperplanes $H_i$ of $M$.
We have $F = D_1$.
Let $I = \{g \in G \mid g \vert_{H_1} = \text{Id}\}$.
Then $\bar G = G/I$ acts faithfully on $H_1$, and $F = H_1/\bar G$.
We have $\tilde F = [H_1/\bar G]$ and $\tilde E = [H_1/G]$. 

We consider a functor $\Phi_0 = \beta_*\alpha^*: D^b(\text{coh}(\tilde F)) \to D^b(\text{coh}(\tilde X))$:

\begin{Lem}
$\Phi_0$ is fully faithful.
\end{Lem}

\begin{proof}
We have $\beta: \tilde E \cong \tilde D_{X,1}$ and 
$\Phi_0(L_F(m)) = \mathcal{O}_{\tilde D_{X,1}}(-\sum_{i \ne 1} m_i\tilde D_{X,i})$.
For $L_F(m) = \mathcal{O}_{\tilde F}(m)$ and $L_F(m') = \mathcal{O}_{\tilde F}(m')$ 
with $m = \sum_{i \ne 1} (m_i/r_it_i)\bar v_i^*$ and $m' = \sum_{i \ne 1} (m'_i/r_it_i)\bar v_i^*$, we have
\[
\begin{split}
&\text{Hom}(L_F(m),L_F(m')) 
\cong H^0(\tilde F, \mathcal{O}_{\tilde F}(m'-m)) \\
&\cong H^0(F, \mathcal{O}_F(-\sum_{i \ne 1} \ulcorner (m_i'-m_i)/r_it_i \urcorner D_{F,i}))
\cong H^0(\tilde D_{X,1}, \mathcal{O}_{\tilde D_{X,1}}(-\sum_{i \ne 1} (m'_i - m_i)\tilde D_{X,i}) \\
&\cong \text{Hom}(\Phi_0(L_F(m)),\Phi_0(L_F(m'))).
\end{split}
\]
\end{proof}

Let $\Lambda = \{k \in \mathbf{Z} \mid 0 \le k < r_1\} \setminus \{k \mid \exists m_1, k = \ulcorner m_1r_1/s_1 \urcorner\}$. 
For each $k \in \Lambda$, we define a fully faithful functor 
$\Phi_k: D^b(\text{coh}(\tilde F)) \to D^b(\text{coh}(\tilde X))$
by $\Phi_k(a) = \Phi_0(a) \otimes \mathcal{O}_{\tilde X}(-k\tilde D_{X,1})$.
We note that 
$\mathcal{O}_{\tilde D_{X,1}}(-r_1\tilde D_{X,1})$ is the image of an invertible sheaf on $\tilde F$ by $\Phi_0$,
so that we put the condition $0 \le k < r_1$.

The assertion of the theorem is proved by the following lemma.
\end{proof}

\begin{Lem}
(1) $\text{Hom}(a,b[p]) = 0$ for all $p$ 
if $a \in \Phi(D^b(\text{coh}(\tilde Y)))$ and $b \in \Phi_k(D^b(\text{coh}(\tilde F)))$.

(2) $\text{Hom}(a,b[p]) = 0$ for all $p$
if $a \in \Phi_k(D^b(\text{coh}(\tilde F)))$ and $b \in \Phi_{k'}(D^b(\text{coh}(\tilde F)))$ with $k < k'$.

(3) $D^b(\text{coh}(\tilde X))$ is generated by $\Phi(D^b(\text{coh}(\tilde Y)))$ and the 
$\Phi_k(D^b(\text{coh}(\tilde F)))$. 
\end{Lem}

\begin{proof}
(1) We have $\alpha_*\mathcal{O}_{\tilde D_{X,1}}(- \sum_i m_i \tilde D_{X,i}) \cong 0$ if $r_1 \not\vert m_1$.
Therefore
\[
\text{Hom}(\mathcal{O}_{\tilde X}(-\sum_i \ulcorner m_ir_i/s_i \urcorner \tilde D_{X,i}), 
\mathcal{O}_{\tilde D_{X,1}}(-k\tilde D_{X,1} - \sum_{i \ne 1} m'_i \tilde D_{X,i})[p]) = 0
\]
for all $p$. 

(2) There are exact sequences
\[
\begin{split}
&0 \to \mathcal{O}_{\tilde X}(-(k+1)\tilde D_{X,1} - \sum_{i \ne 1} m_i \tilde D_{X,i}) \to 
\mathcal{O}_{\tilde X}(-k\tilde D_{X,1} - \sum_{i \ne 1} m_i \tilde D_{X,i}) \\
&\to \mathcal{O}_{\tilde D_{X,1}}(-k\tilde D_{X,1} - \sum_{i \ne 1} m_i \tilde D_{X,i}) \to 0.
\end{split}
\]
Therefore the subcategory generated by $\Phi(D^b(\text{coh}(\tilde Y)))$ and the 
$\Phi_k(D^b(\text{coh}(\tilde F)))$ contains
$\mathcal{O}_{\tilde X}(- \sum_i m_i \tilde D_{X,i})$ for all $m_i$'s, thus coincides
with the whole category.
\end{proof}

\section{known results}

We collect known results which are used in the proof of our result.
We use the terminology of the minimal model program (MMP) as in \cite{KMM}.

We start with the derived McKay correspondence theorem 
for $SL(3,\mathbf{C})$ by Bridgeland-King-Reid:

\begin{Thm}[\cite{BKR}]\label{BKR}
Let $G \subset SL(3,\mathbf{C})$ be a finite subgroup which acts naturally on $M = \mathbf{C}^3$.
Let $X = M/G$ be the quotient space.
Let $Y = G\text{-Hilb}(M)$ be the closed subscheme of the Hilbert scheme of $M$ consisting of 
$G$-invariant $0$-dimensional subschemes $Z \subset M$ 
such that $H^0(\mathcal{O}_Z) \cong \mathbf{C}[G]$ as
$G$-modules.
Let $\mathcal{Z} \subset Y \times M$ be the universal closed subscheme with natural morphisms
$p: \mathcal{Z} \to Y$ and $q: \mathcal{Z} \to M$. 
Then $Y$ is smooth and the natural morphism $f: Y \to X$, called the Hilbert-Chow morphism, 
is a crepant resolution of singularities, i.e., a projective birational morphism such that $K_Y = f^*K_X$.
Moreover there is an equivalence of triangulated categories
\[
\Phi = Rq_*p^*: D^b(\text{coh}(Y)) \to D^b(\text{coh}^G(M)).
\]
\end{Thm}

The following is one of the main results of the Minimal Model Program (MMP):

\begin{Thm}[\cite{BCHM}]\label{BCHM}
Let $(X,B)$ be a KLT pair consisting of a normal variety and an $\mathbf{R}$-divisor, 
and let $f: X \to S$ be a projective morphism to another variety.
Assume that $B$ is big over $S$.
Then there is an MMP for $(X,B)$ over $S$ which terminates to yield a minimal model or a Mori fiber space.
\end{Thm}

We note that $B$ is automatically big if $f$ is a birational morphism.
In this case, there is no possibility of a Mori fiber space.
We also need the following result from the MMP:

\begin{Thm}[\cite{length}]\label{length}
Let $f: X \to Y$ be a contraction morphism from a KLT pair $(X,B)$ in the minimal model program.
Then the exceptional locus of $f$, i.e., the closed subset of $X$ consisting of all 
points where $f$ is not an isomorphism,
is covered by a family of rational curves which are mapped to points by $f$.
\end{Thm}

Let $X$ be a smooth Deligne-Mumford stack whose coarse moduli space is projective.
Then the {\em Hochschild homology} is defined by
$HH_m(X) = \text{Hom}_{X \times X}(\mathcal{O}_{\Delta}, \omega_{\Delta}[\dim X - m])$,
where $\Delta$ denotes the diagonal and $\omega$ is a canonical sheaf.

Hochschild homology is related to the Hodge decomposition:

\begin{Thm}[Hochschild-Kostant-Rosenberg isomorphism \cite{Caldararu}]\label{HKR}
There is an isomorphism 
$HH_m(X) \cong \bigoplus_{p-q=m} H^q(X,\Omega_X^p)$.
\end{Thm}

Hochschild homology is additive for SOD:

\begin{Thm}[\cite{Kuz}]\label{HH}
Let $\tilde X$ and $\tilde Y_i$ ($i=1,\dots,l$) be  smooth Deligne-Mumford stacks whose coarse moduli spaces are projective.
Assume that there are fully faithful functors $\Phi_i: D^b(\tilde Y_i) \to D^b(\tilde X)$ 
such that there is a semi-orthogonal decomposition 
$D^b(\tilde X) = \langle \Phi_1(D^b(\tilde Y_1)), \dots, \Phi_l(D^b(\tilde Y_l)) \rangle$.
Then there is an isomorphism of Hochschild homologies 
$HH_m(\tilde X) \cong \bigoplus_{i=1}^l HH_m(\tilde Y_i)$ for all $m$.
\end{Thm}

\section{maximal $\mathbf{Q}$-factorial terminalization}

There exists a {\em minimal $\mathbf{Q}$-factorial terminalization} for any singularities by \cite{BCHM}.
We define {\em maximal $\mathbf{Q}$-factorial terminalization} for KLT pairs in this section.

\begin{Defn}
Let $(X,B)$ be a KLT pair of a normal variety and an $\mathbf{R}$-divisor.
A projective birational morphism $f: Y \to X$ is said to be a {\em maximal $\mathbf{Q}$-factorial terminalization}
if the following conditions are satisfied:

\begin{itemize}
\item $Y$ has only $\mathbf{Q}$-factorial and terminal singularities.

\item The set of exceptional divisors of $f$ coincides with the set of all prime divisors above $X$ whose
coefficients are non-negative.
\end{itemize}
\end{Defn}

For any normal variety $X$, a minimal $\mathbf{Q}$-factorial terminalization
is a projective birational morphism $f: Y \to X$ such that $Y$ has only $\mathbf{Q}$-factorial and terminal singularities
and that $K_Y$ is $f$-nef.
If $K_X$ is a $\mathbf{Q}$-Cartier divisor, then any exceptional divisor of $f$ has non-negative coefficient.
But we may blow up $Y$ further while preserving the non-negativity of the coefficients (see the example below). 
This is why we use the term \lq\lq maximal''.  

A minimal $\mathbf{Q}$-factorial terminalization is not unique in dimension $\ge 3$.
But they are all $K$-equivalent, i.e., the canonical divisors $K_Y$ are equivalent.
Hence they are connected by flops (\cite{flops}), and we expect that they have equivalent derived categories 
if they are suitably defined.
A maximal $\mathbf{Q}$-factorial terminalization is not unique either in dimension $\ge 3$.
But we cannot expect that they are $K$-equivalent, and their derived categories are not expected to be equivalent either.
Our main theorem is valid for one maximal $\mathbf{Q}$-factorial terminalization, but not necessarily for any
other maximal $\mathbf{Q}$-factorial terminalizations.

\begin{Thm}\label{maximal}
Let $(X,B)$ be a $\mathbf{Q}$-factorial KLT pair of a normal variety and an $\mathbf{R}$-divisor.

\begin{enumerate}

\item[(1)] There exists a maximal $\mathbf{Q}$-factorial terminalization of $(X,B)$.

\end{enumerate}

Let $f: Y \to (X,B)$ be any maximal $\mathbf{Q}$-factorial terminalization. 
Denote $K_Y+B_Y+E_Y = f^*(K_X+B)$, where $B_Y$ is the strict transform of $B$.

\begin{enumerate}

\item[(2)]  There is a week MMP for $-(K_Y+B_Y)$ over $X$ in the following sense: 
$f$ is decomposed into a sequence of birational maps
\[
Y = Y_0 \dashrightarrow Y_1 \dashrightarrow \dots \dashrightarrow Y_l = X
\]
such that each $\alpha_i: Y_{i-1} \to Y_i$ is either a divisorial contraction or a flip for some MMP, and that
$K_{Y_{i-1}} + B_{Y_{i-1}} \le K_{Y_i} + B_{Y_i}$ for all $i$, where
the $B_{Y_i}$ are the strict transforms of $B$.

\item[(3)] Any other maximal $\mathbf{Q}$-factorial terminalizations are connected by flops 
with respect to $K_Y+B_Y+E_Y$ over $X$.

\item[(4)] If there is a reduced divisor $\bar B$ such that $\text{Supp}(B) \subset \bar B$ and 
that $(X,\bar B)$ is a toroidal pair,
then all birational maps $\alpha_i$ are toroidal for suitable toroidal pairs,
and any other maximal $\mathbf{Q}$-factorial terminalizations are connected by toroidal flops.
\end{enumerate}
\end{Thm}

\begin{proof}
(1) Let $g: Z \to X$ be a log resolution, and write $K_Z + B_Z + E_Z =  g^*(K_X+B)$, 
where $B_Z$ is the strict transform of $B$.
By blowing up further if necessary, we may assume that irreducible components of $E_Z$ with non-negative 
coefficients are disjoint each other.
Thus all prime divisors over $X$ with non-negative coefficients already appear on $Z$.
We write $E_Z = E^+_Z - E^-_Z$ with effective divisors $E^+_Z,E^-_Z$ without common irreducible components. 

We run an MMP for the KLT pair $(Z,B_Z + E^+_Z)$ over $X$ using \cite{BCHM}.
Since $K_Z+B_Z + E^+_Z \equiv  E^-_Z$ over $X$, all irreducible components of $E^-_Z$ are contracted during the MMP,
while no other prime divisors are contracted, 
and we obtain a maximal $\mathbf{Q}$-factorial terminalization $f: Y \to X$.

(2) We run an MMP for a KLT pair $(Y, B_Y+(1+\epsilon)E_Y)$ over $X$ for a small positive number $\epsilon$.
Since $K_Y+B_Y + (1+\epsilon)E_Y \equiv \epsilon E_Y \equiv - \epsilon (K_Y+B_Y)$ over $X$, all the irreducible 
components of $E_Y$ are contracted, while the strict transforms of $-(K_Y + B_Y)$ are strictly decreasing during the
MMP.
This step is an MMP for $-(K_Y+B_Y)$.

We contract all other exceptional divisors in the next step which consists of flops and log crepant contractions
for the strict transforms of $K_Y + B_Y$.
Let $f_1: (Y_1,B_{Y_1}) \to X$ be the morphism obtained in the previous step.
We have $K_{Y_1}+B_{Y_1} = f_1^*(K_X+B)$.
Let $F$ be the exceptional locus of $f_1$.
We run an MMP for $(Y_1,B_{Y_1}+\epsilon F)$ over $X$ for small $\epsilon$. 
Then all irreducible components of $F$ are eventually contracted.
Since $K_{Y_1}+B_{Y_1}$ is numerically trivial, 
the strict transforms of $K_{Y_1}+B_{Y_1}$ remain constant.
 
(3) Since $K_Y+B_Y+E_Y$ is numerically trivial over $X$, the assertion follows from \cite{flops}.

(4) We take a toroidal log resolution $g: (Y,\bar C) \to (X,\bar B)$
such that $Y \setminus \bar C \cong X \setminus \bar B$.
Then our MMP over $X$ becomes toroidal automatically. 
The connecting flops are also automatically toroidal.
\end{proof}

\begin{Expl}
Let $X$ be a $2$-dimensional cyclic quotient singularity of type $\frac 1{15}(1,4)$.
We consider a pair $(X,0)$.

The minimal resolution $f_1: Y_1 \to X$ has two exceptional curves $C_1,C_2$ with $C_i^2 = -4$.
We have $K_{Y_1}+\frac 23(C_1+C_2)$.
Blow up $Y_1$ at $C_1 \cap C_2$ to obtain $f_2: Y_2 \to Y_1$ with an exceptional curve $C_3$.
Then blow up $C_3 \cap (C'_1 \cup C'_2)$ to obtain $f_3: Y_3 \to Y_2$ with exceptional curves $C_4,C_5$,
where the symbol ${}'$ denotes the strict transform.
Then the composition $f: Y \to X$ is the maximal resolution.
We have $K_Y =  f^*K_X + \sum_{i=1}^5 e_iE_i$ with $e_1=e_2=\frac 23$, $e_3 = \frac 13$ and
$e_4=e_5=0$, where $E_i=C''_i$ for $i=1,2$, $E_3 = C_3'$ and $E_i=C_i$ for $i=4,5$.

If we choose small positive numbers $\epsilon_i$, then we can run 
an MMP for $(Y, E)$ with $E = \sum_i (e_i+\epsilon_i)E_i$ over $X$.
The final output is always $X$ itself, but there are many possibilities of the order of contractions, some of them are 
via the 
minimal resolution and others not. 
\end{Expl}

For a maximal $\mathbf{Q}$-factorial terminalization $f: Y \to X$ of a quotient singularity, 
we have $K_Y \le K_X$.
Therefore we can expect a semi-orthogonal decomposition according to the DK-hypothesis of the form
$D(X) = \langle C,D(Y) \rangle$.
Moreover we expect that the derived category of a maximal $\mathbf{Q}$-factorial terminalization is further 
decomposed into 
the derived category of a minimal $\mathbf{Q}$-factorial terminalization and its right orthogonal, 
because the 
canonical divisor of the minimal $\mathbf{Q}$-factorial terminalization is in general strictly smaller than that of 
the maximal $\mathbf{Q}$-factorial terminalization.

\section{compactification}

We consider a finite subgroup $G \subset GL(n,\mathbf{C})$ acting naturally on $\mathbf{C}^n$ for arbitrary $n$.
The action of $G$ is extended to a compactification $\mathbf{P}^n$ of $\mathbf{C}^n$.

Let $\rho: G \to End(E_{\rho})$ be an irreducible representation of $G$.
We define a corresponding $G$-equivariant locally free sheaf on $\mathbf{P}^n$ by
$\mathcal{V}_{\rho} = E_{\rho} \otimes_{\mathbf{C}} \mathcal{O}_{\mathbf{P}^n}$, where
the action of $G$ is diagonal, i.e., $g(v \otimes s) = g(v) \otimes (g^{-1})^*(s)$ for 
$v \in E_{\rho}$ and $s \in \Gamma(U,\mathcal{O}_{\mathbf{P}^n})$, where $U$ is a $G$-stable open subset.

The derived category $D^b(\text{coh}([\mathbf{C}^n/G]))$ is generated by their restrictions
$\mathcal{V}_{\rho} \vert_{\mathbf{C}^n}$ for all $\rho$.
We consider $D^b(\text{coh}([\mathbf{P}^n/G]))$ in the following theorem. 

Let $L = \mathbf{P}^n \setminus \mathbf{C}^n$ be the $G$-stable hyperplane at infinity.
We write $\mathcal{V}_{\rho}(iL) = \mathcal{V}_{\rho} \otimes \mathcal{O}_{\mathbf{P}^n}(iL)$.
Let $\pi: \mathbf{P}^n \to [\mathbf{P}^n/G]$ be a natural projection.

\begin{Thm}
(1) $\pi_*\mathcal{O}_{\mathbf{P}^n} \cong \bigoplus_{\rho} \mathcal{V}_{\rho}^{\oplus \dim E_{\rho}}$.

(2) The derived category of the quotient stack $[\mathbf{P}^n/G]$ has a full exceptional collection
consisting of the sheaves $\mathcal{V}_{\rho}(iL)$ for all $\rho$ and $-n \le i \le 0$:
\[
D^b(\text{coh}([\mathbf{P}^n/G])) = \langle \langle \mathcal{V}_{\rho}(-nL) \rangle_{\rho}, \dots, 
\langle \mathcal{V}_{\rho} \rangle_{\rho} \rangle.
\]

(3) $HH_m([\mathbf{P}^n/G]) = 0$ for $m \ne 0$.
\end{Thm}

\begin{proof}
(1) We have a canonical decomposition $\mathbf{C}[G] \cong \bigoplus_{\rho} E_{\rho}^{\oplus \dim E_{\rho}}$.
This decomposition will be globalized.
Since $\pi$ is \'etale, $\pi_*\mathcal{O}_{\mathbf{P}^n}$ is locally free, and 
$\pi^*\pi_*\mathcal{O}_{\mathbf{P}^n} \cong \mathbf{C}[G] \otimes \mathcal{O}_{\mathbf{P}^n}$.
We have $\pi^*\mathcal{V}_{\rho} \cong \mathcal{O}_{\mathbf{P}^n}^{\dim E_{\rho}}$.
Hence
\[
\text{Hom}_{[\mathbf{P}^n/G]}(\mathcal{V}_{\rho}, \pi_*\mathcal{O}_{\mathbf{P}^n}) \cong 
\text{Hom}_{\mathbf{P}^n}(\pi^*\mathcal{V}_{\rho}, \mathcal{O}_{\mathbf{P}^n}) \cong \mathbf{C}^{\dim E_{\rho}}.
\]
Therefore we obtain our result.

(2) 
We have $D^b(\mathbf{P}^n) = \langle \mathcal{O}_{\mathbf{P}^n}(-nL), \dots, \mathcal{O}_{\mathbf{P}^n}\rangle$ by
Beilinson \cite{Beilinson}.
We have a canonical decomposition 
\[
\pi_*\mathcal{O}_{\mathbf{P}^n}(iL) \cong \bigoplus_{\rho} \mathcal{V}_{\rho}(iL)^{\dim E_{\rho}}.
\]
Since $\pi$ is \'etale, we have 
$\text{Hom}_{[\mathbf{P}^n/G]}(\pi_*\mathcal{O}_{\mathbf{P}^n}(i),a) \cong \text{Hom}_{\mathbf{P}^n}(\mathcal{O}_{\mathbf{P}^n}(i),\pi^*a)$
for $a \in D^b([\mathbf{P}^n/G])$.
Therefore $D^b([\mathbf{P}^n/G])$ is generated by the $\mathcal{V}_{\rho}(iL)$ for all $\rho$ and $-n \le i \le 0$.

If $-n \le j-i < 0$, we have
\[
\text{Hom}_{[\mathbf{P}^n/G]}(\mathcal{V}_{\rho}(iL), \mathcal{V}_{\rho'}(jL)[m]) 
= (E_{\rho^*} \otimes E_{\rho'} \otimes H^m(\mathbf{P}^n,\mathcal{O}_{\mathbf{P}^n}((j-i)L)))^G 
= 0
\]
for all $m$.
On the other hand, we have
\[
\text{Hom}_{[\mathbf{P}^n/G]}(\mathcal{V}_{\rho}(iL), \mathcal{V}_{\rho'}(iL)[m]) 
= (E_{\rho^*} \otimes E_{\rho'} \otimes H^m(\mathbf{P}^n,\mathcal{O}_{\mathbf{P}^n}))^G 
= \begin{cases}
\mathbf{C} &\text{ if } \rho = \rho' \text{ and } m=0 \\ 
0 &\text{ otherwise}.
\end{cases}
\]
Therefore $(\mathcal{V}_{\rho}(iL))$ with a suitable order is an exceptional collection.

(3) This is due to the additivity of the Hochschild homology (Theorem~\ref{HH}).
\end{proof}

\section{proof of Theorem~\ref{main}}

{\em Step 1}.

Let $H = G \cap SL(3,\mathbf{C})$.
Then $H$ is a normal subgroup of $G$, and $G/H \cong \mathbf{Z}_r$ is a cyclic group of order $r = [G:H]$.
By Theorem~\ref{BKR} (\cite{BKR}), an irreducible component of the Hilbert scheme of 
$H$-equivariant sheaves on $\mathbf{C}^3$, denoted by $Y_1 = H\text{-Hilb}(\mathbf{C}^3)$, 
gives a resolution of singularities $f_1: Y_1 \to \mathbf{C}^3/H$ such that 
$K_{Y_1} = f_1^*K_{\mathbf{C}^3/H}$.
Moreover there is an equivalence
$D^b([\mathbf{C}^3/H]) \cong D^b(Y_1)$ of triangulated categories.

By construction, the quotient group $\mathbf{Z}_r$ acts on $Y_1$.
Let $\pi_1: Y_1 \to X_1 = Y_1/\mathbf{Z}_r$ be the natural projection, 
and $\bar f_1: X_1 \to X = \mathbf{C}^3/G$
the induced projective birational morphism.
We define a $\mathbf{Q}$-divisor $B_1$ on $X_1$ by 
$\pi_1^*(K_{X_1}+B_1) = K_{Y_1}$.
The coefficients of $B_1$ belong to the standard set $\{1-1/e \mid e \in \mathbf{Z}_{>0}\}$.
The quotient pair $(X_1,B_1)$ has only toroidal KLT singularities and  in the following sense:
$X_1$ has an \'etale covering $\{U\}$ with reduced divisors $\bar B_U$ such that 
$\text{supp}(B_1) \cap U \subset \bar B_U$ and that the pairs $(U,\bar B_U)$ are toroidal without self-intersection. 

Let $\pi_2: \mathbf{C}^3/H \to X$ be the quotient map and define a $\mathbf{Q}$-divisor $B$ on $X$ by 
$\pi_2^*(K_X+B) = K_{\mathbf{C}^3/H}$.
Since $f_1^*K_{\mathbf{C}^3/H} = K_{Y_1}$ and $\pi_2^*(K_X+B)=K_{\mathbf{C}^3/H}$, 
we have $\bar f_1^*(K_X+B)=K_{X_1}+B_1$.
Here we note that $H$ does not contain quasi-reflections, but $G$ may do so that $B$ may be non-zero.
We have the following commutative diagram:
\[
\begin{CD}
@. @. Y \\
@. @. @VVgV \\
@. Y_1 @>{\pi_1}>> X_1 \\
@. @V{f_1}VV @VV{\bar f_1}V \\
\mathbf{C}^3 @>{\pi_3}>> \mathbf{C}^3/H @>{\pi_2}>> X
\end{CD}
\]

We take a maximal $\mathbf{Q}$-factorial terminalization $g: Y \to (X_1,B_1)$ by Theorem~\ref{maximal}.
It is automatically locally toroidal, i.e.,
toroidal on each $U$ with respect to $(U,\bar B_U)$. 
Hence $Y$ has only isolated terminal quotient singularities.
By classification, they are quotient singularities of types $\frac 1{r_i}(a_i,-a_i,1)$ for coprime integers $r_i,a_i$ 
with $0 < a_i < r_i$. 
We write $g^*(K_{X_1}+B_1) = K_Y + B_Y+E_Y$, where $B_Y$ is the strict transform of $B_1$.

Let $f = \bar f_1 \circ g: Y \to X$.
Then we have $f^*(K_X+B) = K_Y + B_Y+E_Y$.
Since $Y$ is a highest model which satisfies this equality with $B_Y+E_Y$ effective,
$f$ is also a maximal $\mathbf{Q}$-factorial terminalization of $(X,B)$.

\vskip 1pc

{\em Step 2}.

We obtained equalities of log canonical divisors in the previous step.
We check the corresponding equivalences and fully faithful embeddings of the derived categories.

First by Theorem~\ref{BKR}, we have $D^b([\mathbf{C}^3/H]) \cong D^b(Y_1)$.
Since $\mathbf{Z}_r$ acts on $[\mathbf{C}^3/H]$ and $Y_1$ compatibly, we have 
$D^b([\mathbf{C}^3/G]) \cong D^b([Y_1/\mathbf{Z}_r])$.
We note that the Deligne-Mumford stack $[Y_1/\mathbf{Z}_r]$ coincides with 
the one associated to the pair $(X_1,B_1)$.

By Theorem~\ref{maximal}, $g: Y \to X_1$ is decomposed into a sequence of locally toroidal birational maps
\[
Y = Y_0 \dashrightarrow Y_1 \dashrightarrow \dots \dashrightarrow Y_m = X_1
\]
such that each $\alpha_i: Y_{i-1} \to Y_i$ is either a divisorial contraction or a flip, and that
$K_{Y_{i-1}} + B_{Y_{i-1}} \le K_{Y_i} + B_{Y_i}$ for all $i$, where
the $B_{Y_i}$ are the strict transforms of $B_1$.

We apply Theorem~\ref{toric} to the following inequalities of the log canonical divisors:
\[
K_Y \le K_Y + B_Y = K_{Y_0}+B_{Y_0} \le \dots \le K_{Y_m}+B_{Y_m} = K_{X_1}+B_1.
\]
For each inequality, several semi-orthogonal components are cut out from $D^b([Y_1/\mathbf{Z}_r])$, and 
each component is equivalent to the derived category of a Deligne-Mumford stack above certain center, 
i.e., an irreducible component of $B_Y$ for 
the first inequality, and an irreducible component of the image of the exceptional locus of the birational map 
for the rest.
If the Deligne-Mumford stacks are associated to pairs, then they are further semi-orthogonally decomposed 
by the first case of Theorem~\ref{toric}.
Thus we obtain toroidal varieties $Z_1,\dots,Z_{l'}$ proper above $X_1$
such that there is a semi-orthogonal decomposition
\[
D^b([Y_1/\mathbf{Z}_r]) = \langle D^b(\tilde Z_1), \dots, D^b(\tilde Z_{l'}), D^b(\tilde Y) \rangle
\]
where the $\tilde Z_i$ are the smooth Deligne-Mumford stacks associated to the $Z_i$.

We note that, although the toroidal structures are defined only \'etale locally over $X_1$, the centers $Z_j$ are
defined globally above $X_1$.
Indeed if some $Z_j$ is above a curve on $X_1$, then the maximal $\mathbf{Q}$-factorial terminalization
along generic points of the curve is unique, hence they glue together.

Let $d_i = \dim Z_i \le 2$.
If $d_i = 0$, then $Z_i$ is a point above the origin of $X$.
In this case $D^b(\tilde Z_i)$ is generated by an exceptional object.

If $d_i = 1$, then $Z_i$ is a smooth curve.
It is either projective above the origin, or projective
above the $1$-dimensional singular locus of $X$.
We shall prove that $Z_i$ is a rational curve in either case in the next step.
Therefore if $Z_i$ is projective, then $D^b(\tilde Z_i)$ is further semi-orthogonally 
decomposed into the derived categories of points.

If $d_i = 2$, then $Z_i$ is a normal surface with only quotient singularities.
It is either projective above the origin, projective
above the $1$-dimensional singular locus of $X$, or above the image of a hyperplane which corresponds to
a quasi-reflection in $G$. 
In the first case, $Z_i$ is covered by rational curves by Theorem~\ref{length}.
Therefore $D^b(\tilde Z_i)$ is further semi-orthogonally 
decomposed into the derived categories of curves and points.
In the second case, $Z_i$ is covered by rational curves contained in the fibers over points of $X_1$, 
and $D^b(\tilde Z_i)$ is again further semi-orthogonally decomposed into the derived categories of curves and points.
In the last case, $D^b(\tilde Z_i)$ is semi-orthogonally decomposed into the derived categories of the minimal 
resolution of the quotient surface and points.
We shall also prove that the curves appearing in these decompositions are all rational in the next step.

\vskip 1pc

{\em Step 3}. 

The curves appearing in Step 2 are compactified to smooth projective curves above $\mathbf{P}^3$, 
and also appear in the parallel SOD of $D^b([\mathbf{P}^n/G])$. 
Suppose that one of these curves $C$ has a positive genus.
By Theorem~\ref{HKR}, we have $HH_1(C) \cong H^0(C, \Omega_C^1) \ne 0$, a contradiction.
Therefore only rational curves appear.

Graduate School of Mathematical Sciences, University of Tokyo,

Komaba, Meguro, Tokyo, 153-8914, Japan

kawamata@ms.u-tokyo.ac.jp

\end{document}